\documentclass[10pt,draft,reqno]{amsart}

     \makeatletter
     \def\section{\@startsection{section}{1}%
     \z@{.7\linespacing\@plus\linespacing}{.5\linespacing}%
     {\bfseries
     \centering
     }}
     \def\@secnumfont{\bfseries}
     \makeatother
\newtheorem{theorem}{Theorem}[section]
\newtheorem{lemma}[theorem]{Lemma}
\newtheorem{proposition}[theorem]{Proposition}
\newtheorem{corollary}[theorem]{Corollary}
\theoremstyle{definition}
\newtheorem{definition}[theorem]{Definition}
\newtheorem{example}[theorem]{Example}
\theoremstyle{remark}
\newtheorem{remark}[theorem]{Remark}
\numberwithin{equation}{section} 

\def\bR{\mathbb{R}}
\def\bZ{\mathbb{Z}}
\def\bE{\mathbb{E}}

\def\bN{\mathbb{N}}
\def\bD{\mathbb{D}}

\def\cB{\mathcal{B}}

\def\cD{\mathcal{D}}
\def\cE{\mathcal{E}}
\def\cF{\mathcal{F}}
\def\cG{\mathcal{G}}
\def\cH{\mathcal{H}}
\def\cK{\mathcal{K}}
\def\cP{\mathcal{P}}

\def\cO{\mathcal{O}}

\def\cS{\mathcal{S}}

\begin{document}

\title[Stochastic Heat Equation: germ Markov Property]
{The Stochastic Heat Equation Driven by a Gaussian Noise: germ
Markov Property}

\author[Raluca Balan]{Raluca Balan*}
\thanks{The first author was supported by a grant from the Natural Sciences and
Engineering Research Council of Canada. The second author was
supported by a post-doctoral fellowship at University of Ottawa.}
\address{Corresponding author. Department of Mathematics and Statistics,
University of Ottawa, 585, King Edward Avenue, Ottawa, Ontario K1N
6N5 Canada} \email{rbalan@uottawa.ca}
\urladdr{http://aix1.uottawa.ca/$\sim$rbalan}

\author{Doyoon Kim}
\address{Department of Mathematics,
University of Southern California, Los Angeles, 3620 Vermont Avenue,
KAP108, Los Angeles, CA 90089, USA} \email{doyoonki@usc.edu}
\urladdr{http://college.usc.edu/faculty/faculty1017406.html}



\begin{abstract}
\noindent Let $u=\{u(t,x);t \in [0,T], x \in {\mathbb{R}}^{d}\}$ be
the process solution of the stochastic heat equation $u_{t}=\Delta
u+ \dot F, u(0,\cdot)=0$ driven by a Gaussian noise $\dot F$, which
is white in time and has spatial covariance induced by the kernel
$f$. In this paper we prove that the process $u$ is locally germ
Markov, if $f$ is the Bessel kernel of order $\alpha=2k,k \in
\bN_{+}$, or $f$ is the Riesz kernel of order $\alpha=4k,k \in
\bN_{+}$.
\end{abstract}


\maketitle

\section{Introduction}

This article is based on the theory of stochastic partial
differential equations (s.p.d.e.'s) initiated by John Walsh in 1986 \cite{walsh86}.
This theory relies on the construction of a stochastic integral with
respect to the so-called ``worthy martingale measures'' and focuses
mainly on equations driven by a space-time white noise. 

Recently, there has been a considerable amount of interest in
s.p.d.e.'s driven by a noise term which is white in time, but is
``colored'' in space, in the sense that it has a spatial covariance
structure induced by a kernel $f$, which is the Fourier transform of
a tempered distribution $\mu$. This line of research was initiated
in \cite{mueller97}, \cite{dalang-frangos98} for the stochastic wave
equation in the case $d=2$, and then generalized in \cite{dalang99},
\cite{dalang-mueller03} to a larger class of s.p.d.e.'s in arbitrary
spatial dimensions. The new theory requires an extension of Walsh's
stochastic integral, to include the case when the integrand is a
(Schwartz) distribution; this type of extension is needed for
instance for the stochastic wave equation with $d \geq 3$. Under a
certain integrability condition imposed on the measure $\mu$, one
can prove that any second-order s.p.d.e. with constant coefficients
has a process solution, i.e. a solution which can be identified with
a multiparameter stochastic process $u$. As far as we know, the
literature to date does not contain any study of the germ Markov
property of this process solution. Such a study is of great
importance because it allows us to conclude that the behavior of the
process in any time-space region $A$ is independent of its behavior
outside the region, given the values of the process in a thin area
around the boundary $\partial A$ of the region.

There is an immense amount of literature dedicated to the study of
different types of Markov properties for multiparameter processes
(see 
\cite{balan-ivanoff02} and the references
therein). In particular, the germ Markov property of multiparameter
Gaussian processes is an area which received a lot of attention in
the 1970's, which allows for the use of various tools imported from
Hilbert space analysis, like the method of the Reproducing Kernel
Hilbert Space (RKHS). An excellent all-time reference is the
monograph \cite{rozanov82}; earlier references include
\cite{pitt71}, \cite{kunsch79}, \cite{nualart-sanz79}.

In the 1990's, the problem of the germ Markov property for
multi-dimensional Gaussian processes appears again in the
literature, this time for processes which arise as solutions of
s.p.d.e.'s driven by a white noise: the case of elliptic equations
was thoroughly treated in \cite{donatimartin-nualart94}, while the
systematic study of this problem for the quasi-linear parabolic
equations (in the case $d=1$) is found in \cite{nualart-pardoux94}.
The case of hyperbolic equations turns out to be the most difficult
one: the only reference here seems to be \cite{dalang-hou97}, in
which the authors investigate very carefully the structure of the
germ $\sigma$-fields induced by the process solution of the
stochastic wave equation (in the case $d=2$) driven by a L\'{e}vy
noise without Gaussian component. (This type of noise induces a very
particular form for the process solution, which is used in a
fundamental way for deriving its germ Markov property.) Finally, the
more recent work \cite{pitt-robeva03} contains a detailed analysis
of the relationship between the germ and sharp $\sigma$-fields
(which are used for defining the sharp and germ Markov property,
respectively), in the case of the process solution associated to the
Bessel equation driven by a white noise (in the case $d=2$).

The goal of the present paper is to investigate the structure of the
RKHS associated to the process solution $u$ of the stochastic heat
equation driven by a spatial covariance kernel $f$, and to identify
some examples of functions $f$ for which the process $u$ possesses
the germ Markov property. As far as we know, this is the first
attempt to tackle a problem of this type in the literature. Our two
main results (Theorem \ref{main-theorem-Bessel} and Theorem
\ref{main-theorem-Riesz}) state that the process solution $u$ has
the germ Markov property, if the function $f$ is either the Bessel
kernel of order $\alpha=2k, k \in {\bf N}$ or the Riesz kernel of
order $\alpha=4k, k \in {\bf N}$. In order to prove these results,
we use some results from the $L_{p}$-theory of parabolic equations
with mixed norms, due to \cite{krylov99}, \cite{krylov01}. The germ
Markov property of the process solution $u$ in the case of the
Bessel or Riesz kernels with $\alpha>0$ arbitrary (or in the case of
other kernel functions) remains an open problem.


The paper is organized as follows. In Section 2, we introduce the
framework for the study of s.p.d.e.'s, including the construction of
the extended stochastic integral due to Dalang (as in
\cite{dalang99}). In Section 3, we include all the ingredients which
are necessary to formulate the result about the existence of the
process solution $u$, and we examine the structure of the Gaussian
space and the RKHS associated to this process. In Section 4, we give
the definition of the germ Markov property and we prove our two main
results, by applying a fundamental result of \cite{kunsch79}. The
two appendices contain some technical proofs.

\section{The Framework}
\subsection{Basic Notation}

We denote by $\cD(U)$ the space of all infinitely differentiable
functions on $\mathbb{R}^{n}$ whose support is compact and contained
in the open set $U$, and by $\cD'(U)$ the space of {\em
distributions}, i.e. continuous linear functionals on $\cD(U)$. We
denote by $\cS(\mathbb{R}^{n})$ the Schwartz space of all rapidly
decreasing functions on $\mathbb{R}^{n}$, and by
$\cS'(\mathbb{R}^{n})$ the space of {\em tempered distributions},
i.e. continuous linear functionals on $\cS(\mathbb{R}^{n})$. The
space $\cS'(\mathbb{R}^{n})$ can be viewed as a subspace of
$\cD'(\mathbb{R}^{n})$. An important subspace of
$\cS'(\mathbb{R}^{n})$ is the space $\cO_{C}'(\mathbb{R}^{n})$ of
all distributions with {\em rapid decrease} (see Chapter VII,
section 5, \cite{schwartz66}).

For an arbitrary function $\phi$ on $\mathbb{R}^{n}$, the
translation by $x$ is denoted with $\phi_{x}$, i.e.
$\phi_{x}(y)=\phi(x+y)$, and the reflection by $0$ is denoted with
$\tilde \phi$, i.e. $\tilde \phi(x)=\phi(-x)$. These notions have
obvious extensions to distributions. The Fourier transform of $\phi
\in L_{1}(\mathbb{R}^{n})$ is defined by $\cF\phi (\xi)=(2
\pi)^{-n/2}\int_{\mathbb{R}^{n}} \exp(- i \xi \cdot x) \phi(x)dx$.
The map $\cF:\cS(\mathbb{R}^{n}) \rightarrow \cS(\mathbb{R}^{n})$
can be extended to $L_{2}(\mathbb{R}^{n})$ and
$\cS'(\mathbb{R}^{n})$. For every subset $U \subset \bR^n$, we
denote
$$ \langle \varphi,\psi \rangle_{L_2(U)} = \int_{U} \varphi(x)
\overline{\psi(x)} dx, $$ whenever the integral is defined.

\subsection{The Gaussian Noise}

\noindent Let $T>0$ and $f$ be a locally integrable function on
$\mathbb{R}^{d}$. As in \cite{dalang99}, for every $\varphi, \psi
\in \cD((0,T) \times \mathbb{R}^{d})$ we define
\begin{equation}
\label{def-inner-product} J_{f}(\varphi, \psi)=\langle \varphi,\psi
\rangle_{0}=\int_{0}^{T} \int_{\bR^{d}} \int_{\bR^{d}}
\varphi(t,x)f(x-y) \psi(t,y)dy dx dt.
\end{equation}

 The Gaussian
 noise mentioned in the Introduction will be a zero-mean Gaussian
 process with covariance $J_{f}$. The existence of this process is based on the
following Bochner-Schwartz type result (see Theorem 2, p. 157,
\cite{gelfand-vilenkin64}):

\begin{lemma}
\label{3-equivalences} The bi-functional $J_{f}$ is
nonnegative-definite if and only if there exists a tempered measure
$\mu$ on $\mathbb{R}^{d}$ such that
$$f(x)=\int_{\mathbb{R}^{d}}e^{-i \xi \cdot x}\mu(d \xi), \ \ \ \forall
x \in R^{d}.$$
\end{lemma}

In this case, for every $\varphi, \psi \in \cD((0,T) \times
\mathbb{R}^{d})$ $$ \langle \varphi,\psi \rangle_{0}=\int_0^T
\int_{\bR^d} \cF\varphi(t,\xi)
\overline{\cF\psi(t,\xi)}\mu(d\xi)dt,$$ where $\cF \varphi(t,\cdot)$
denotes the Fourier transform of $\varphi (t,\cdot)$.

The basic example of kernel functions is the white noise kernel:
$f(x)=\delta(x)$, $\mu(d\xi)=d\xi$.

More interesting examples of kernel functions $f$ arise when
$\mu(d\xi)=g(|\xi|^{2})d\xi$, for a certain function $g$ on
$(0,\infty)$. In this case, $f$ becomes the inverse Fourier
transform of the tempered distribution $\xi \mapsto g(|\xi|^{2})$
and we can define the operator $g(-\Delta):\cS(\mathbb{R}^{d})
\rightarrow \cS'(\mathbb{R}^{d})$ by
$\cF[g(-\Delta)\phi](\xi)=g(|\xi|^{2})\cF \phi(\xi)$, or
equivalently
$$
g(-\Delta)\phi=\phi *f.
$$
(see p. 149-152, \cite{folland95}, or p. 117-132, \cite{stein70}).
Here are some typical examples:

\begin{example}
\label{riesz} {\rm If $\mu(d\xi)=|\xi|^{-\alpha}d\xi$, then $f$ is
the Riesz kernel of order $\alpha$:
$$f(x)=R_{\alpha}(x):=
\gamma_{\alpha}|x|^{-d+\alpha}, \ \ \ 0<\alpha<d,$$
 where $\gamma_{\alpha}$ is an appropriate constant.
In this case, for every $\phi \in \cS(\mathbb{R}^{d})$ we have
$(-\Delta)^{-\alpha/2}\phi=\phi*R_{\alpha}$.}
\end{example}

\begin{example}
\label{bessel}{\rm If $\mu(d\xi)=(1+|\xi|^{2})^{-\alpha/2}d\xi$,
then $f$ is the Bessel kernel of order $\alpha$:
$$f(x)=B_{\alpha}(x):=
c_{\alpha}\int_{0}^{\infty}\tau^{(\alpha-d)/2-1}
e^{-\tau-|x|^{2}/(4\tau)} d\tau, \ \ \ \alpha>0,$$ where
$c_{\alpha}$ is an appropriate constant. In this case, for every
$\phi \in \cS(\mathbb{R}^{d})$ we have
$(1-\Delta)^{-\alpha/2}\phi=\phi*B_{\alpha}$.}
\end{example}

\begin{example} {\rm If $\mu(d\xi)=e^{-4\pi^2 \alpha|\xi|^{2}}$, then
$f$ is the heat kernel
$$f(x)=G_{\alpha}(x):= (4\pi \alpha)^{-d/2}
e^{-|x|^{2}/(4\alpha)}, \ \ \ \alpha>0.$$ In this case, for every
$\phi \in \cS(\mathbb{R}^{d})$ we have
$e^{\alpha\Delta}\phi=\phi*G_{\alpha}$.}
\end{example}

In what follows, we let  $F=\{F(\varphi);\varphi \in \cD((0,T)
\times \mathbb{R}^{d})\}$ be a zero-mean Gaussian process with
covariance $J_{f}$, i.e. $\forall \varphi,\psi \in \cD((0,T) \times
\mathbb{R}^{d})$
\begin{equation}
\label{covariance} E(F(\varphi) F(\psi))=\langle \varphi,\psi
\rangle_{0}.
\end{equation}
The Gaussian space $H^{F}$ of the process $F$ is defined as the
closed linear subspace of $L_{2}(\Omega)$, generated by the
variables $\{F(\varphi),\varphi \in \cD((0,T) \times
\mathbb{R}^{d})\}$.

\subsection{The Stochastic Integral}

In this subsection, we summarize the construction of the generalized
stochastic integral $M(\varphi)=\int_{0}^{T}\int_{\mathbb{R}^{d}}
\varphi(t,x)M(dt,dx)$ with respect to the martingale measure $M$
induced by the noise $F$, due to Dalang (see \cite{dalang99}). For
our purposes, it is enough to consider only the case of
deterministic integrands. Here is the construction procedure:

\vspace{2mm}

{\em Step 0.} For each $\varphi \in \cD((0,T) \times R^{d})$, we set
$M(\varphi)=F(\varphi) \in H^{F}$.

\vspace{2mm}

 {\em Step 1.} Let ${\cB}_{b}(\mathbb{R}^{d})$ be the class of all bounded Borel subsets
 of $\mathbb{R}^{d}$ and $\cE^{(d)}$ be the class
of all linear combinations of functions $1_{[0,t] \times A}$, $t \in
[0,T], A \in {\cB}_{b}(\mathbb{R}^{d})$. We endow $\cE^{(d)}$ with
the inner product $\langle \cdot,\cdot \rangle_{0}$ given by formula
(\ref{def-inner-product}) and we denote by $\| \cdot \|_{0}$ the
corresponding norm.

For each $t \in (0,T), A \in {\cB}_{b}(\mathbb{R}^{d})$, there
exists a sequence $\{\varphi_{n}\}_{n} \subset \cD((0,T) \times
\mathbb{R}^{d})$ such that $\varphi_{n} \to 1_{[0,t] \times A}$ and
${\rm supp} \varphi_{n} \subseteq K$ for all $n$ (see
\cite{dalang-frangos98}, p. 190).
By the bounded convergence theorem, $\| \varphi_{n}-1_{[0,t] \times
A}\|_{+} \to 0$, where $\| \cdot \|_{+}$ is the norm defined in Step
2 below. Hence $\| \varphi_{n}-1_{[0,t] \times A}\|_{0} \to 0$ and
$\bE( M(\varphi_{m})-M(\varphi_{n}))^2= \| \varphi_{m}-\varphi_{n}
\|_{0} \to 0$ as $m,n \rightarrow \infty$, i.e. the sequence
$\{M(\varphi_{n})\}_{n}$ is Cauchy in $L_{2}(\Omega)$. A standard
argument shows that its limit does not depend on
$\{\varphi_{n}\}_{n}$. We set $M_{t}(A)=M(1_{[0,t] \times
A})=_{L_{2}(\Omega)}\lim_{n}M(\varphi_{n}) \in H^{F}$.
 (Note that the
process $M=\{M_{t}(A); t \in [0,T],A \in {\cB}(\mathbb{R}^{d}) \}$
is a worthy martingale measure, in the sense of \cite{walsh86}.) We
extend $M$ by linearity to $\cE^{(d)}$. A limiting argument and
relation (\ref{covariance}) shows that, for every $\varphi,\psi \in
{\cE}^{(d)}$
\begin{equation}
\label{covariance-1} \bE(M(\varphi)M(\psi))=\langle \varphi,\psi
\rangle_{0}.
\end{equation}

\vspace{2mm}

 {\em Step 2.} Let $\cP_{+}^{(d)}=\{\varphi:[0,T] \times
\mathbb{R}^{d} \rightarrow \mathbb{R} \ {\rm measurable} ; \|
\varphi \|_{+}<\infty\}$, where $\| \varphi \|_{+}^{2}:=
 \int_{0}^{T} \int_{\mathbb{R}^{d}} \int_{\mathbb{R}^{d}}
|\varphi(t,x)|f(x-y) |\varphi(t,y)|dy dx dt$. We endow
$\cP_{+}^{(d)}$ with the inner product $\langle \cdot, \cdot
\rangle_{0}$ given by formula (\ref{def-inner-product}) and we
denote by $\| \cdot \|_{0}$ the corresponding norm; note that $\|
\cdot \|_{0} \leq \| \cdot \|_{+} $. An argument similar to that
used in the proof of Proposition 2.3, \cite{walsh86}, shows that
$\cE^{(d)}$ is dense in $\cP_{+}^{(d)}$ with respect to $\| \cdot
\|_{+}$, and hence for each $\varphi \in \cP_{+}^{(d)}$,  there
exists a sequence $\{ \varphi_{n} \}_{n} \subset \cE^{(d)}$ such
that $\| \varphi_{n}-\varphi
\|_{+} \to 0$. As in Step 1, 
we set $M(\varphi)=_{L_{2}(\Omega)}\lim_{n}M(\varphi_{n}) \in
H^{F}$. Note that relation (\ref{covariance-1}) holds for every
$\varphi,\psi \in {\cP}_{+}^{(d)}$.

\vspace{2mm}

{\em Step 3.}  Let $\cP_{0}^{(d)}$ be the completion of $\cE^{(d)}$
with respect to $< \cdot, \cdot>_{0}$. The space $\cP_{0}^{(d)}$ is
the largest space of integrands $\varphi$ for which we can define
the stochastic integral $M(\varphi)$. According to
\cite{dalang99},p. 9, the space $\cP_{0}^{(d)}$ has the following
alternative definition.

 Let $\overline{{\cP}}^{(d)}=\{\varphi:[0,T]
\to {\cS}'(\mathbb{R}^{d}) ; \ {\cF}\varphi(t, \cdot) \
\mbox{function} \forall t \in [0,T], \ (t,\xi) \mapsto
{\cF}\varphi(t,\xi) \linebreak \mbox{ measurable}, \| \varphi
\|_{0}<\infty\}$, where $\| \varphi \|_{0}^{2}=\int_{0}^{T}
\int_{\mathbb{R}^{d}} |{\cF}\varphi(t,\xi)|^{2}\mu(d\xi)dt$. Let
${\cE}_{0}^{(d)}=\{\varphi \in {\cP}_{+}^{(d)}; \varphi(t, \cdot)
\in {\cS}(\mathbb{R}^{d}), \ \forall t \in [0,T]\}$ and
$\cP_{0}^{(d)}$ be the closure of $\cE_{0}^{(d)}$ in
$\overline{\cP}^{(d)}$ with respect to $\| \cdot \|_{0}$.
Note that 
$$\langle \varphi,\psi \rangle_{0}=\int_{0}^{T} \int_{\mathbb{R}^{d}}{\cF}
\varphi(t,\xi) \overline{{\cF} \psi(t,\xi)}\mu(d\xi)dt, \  \ \
\varphi,\psi \in {\cP}_{0}^{(d)}.$$

\noindent For each $\varphi \in {\cP}_{0}^{(d)}$, there exists a
sequence $\{ \varphi_{n} \}_{n} \subset {\cE}_{0}^{(d)}$ such that
$\| \varphi_{n}-\varphi \|_{0} \rightarrow 0$. As in Step 1,
we set $M(\varphi)=_{L_{2}(\Omega)}\lim_{n}M(\varphi_{n}) \in
H^{F}$. Note that relation (\ref{covariance-1}) holds for every
$\varphi,\psi \in {\cP}_{0}^{(d)}$, and hence
\begin{equation}
\label{M-HF} \varphi \mapsto M(\varphi) \ \mbox{is an isometry
between} \ {\cP}_{0}^{(d)} \ \mbox{and} \
H^{F}.
\end{equation}

\vspace{3mm}

In summary, the previous construction is based on the diagram
\begin{eqnarray*}
{\cD}((0,T) \times \mathbb{R}^{d}) \subset {\cE}_{0}^{(d)} \subset &
{\cP}_{+}^{(d)} & \subset {\cP}_{0}^{(d)}
\subset \overline{\cP}^{(d)} \\
& \cup & \\
& {\cE}^{(d)} &
\end{eqnarray*} and the following
$3$ approximation techniques:
\begin{itemize}
\item indicator functions can be approximated by functions in
${\cD}((0,T) \times \mathbb{R}^{d})$;

\item functions in ${\cP}_{+}^{(d)}$ can be approximated by
indicator functions;

\item {\em distributions} in ${\cP}_{0}^{(d)}$ can be
approximated by ``smooth'' functions in ${\cP}_{+}^{(d)}$.
\end{itemize}

\noindent In particular, we conclude that $\cD((0,T) \times
\mathbb{R}^{d})$ is dense in ${\cP}_{0}^{(d)}$ with respect to $\|
\cdot \|_{0}$.



\subsection{Alternative Characterization of the Space
$\cP_{0}^{(d)}$}

As in \cite{ferrante-sanzsole06}, for every $\varphi,\psi \in
{\cD}(\bR^d)$, define
$$\langle \varphi,\psi
\rangle_{0,x}:=\int_{\mathbb{R}^{d}}\int_{\mathbb{R}^{d}}\varphi(x)f(x-y)\psi(y)
dydx= \int_{\bR^d} \cF \varphi(\xi) \overline{\cF \psi(\xi) } \, \mu
(d\xi),$$ and let $\cP_{0,x}^{(d)}$ be the completion of
$\cD(\bR^{d})$ with respect to $\langle \cdot, \cdot \rangle_{0,x}$.
Note that
\begin{equation}
\label{product0-produxt0x} \langle \varphi,\psi
\rangle_{0}=\int_{0}^{T} \langle \varphi(t,\cdot),\psi(t,\cdot)
\rangle_{0,x}dt, \ \ \ \forall \varphi,\psi \in {\cP}_{0}^{(d)}.
\end{equation}

\begin{remark}
\label{P0d-in-L2}
{\rm Note that for every $\varphi \in \cP_{0}^{(d)}$, $\varphi(t,
\cdot) \in \cP_{0,x}^{(d)}$ for a.e. $t \in [0,T]$. Since
$\cE^{(d)}$ is dense in $\cP^{(d)}$, one can prove that the map $t
\mapsto \varphi(t, \cdot)$ is strongly measurable from $[0,T]$ to
$\cP_{0,x}^{(d)}$ (in the sense of Definition on p.649,
\cite{evans98}). Using (\ref{product0-produxt0x}), we conclude that
${\cP}_{0}^{(d)} \subset L_{2}((0,T), {\cP}_{0,x}^{(d)})$ and $\|
\varphi \|_{0}=\| \varphi \|_{L_{2}((0,T),\cP_{0,x}^{(d)})}$,
$\forall \varphi \in \cP_{0}^{(d)}$.}
\end{remark}

\section{The Process Solution}
\subsection{The Equation and its Solution}

We consider the stochastic heat equation with vanishing initial
conditions, written {\em formally} as:

\begin{equation}
\label{heat-eq}u_{t}-\Delta u =  \dot F, \ \ \ \mbox{in} \ (0,T)
\times \mathbb{R}^{d}, \ \ \  u(0,\cdot)= 0.
\end{equation}

\noindent The solution of this equation is defined formally as
follows. Suppose for the moment that $\dot F$ is a random variable
with values in $\cS'(\mathbb{R}^{d+1})$. For every fixed $\omega \in
\Omega$,
 let $\{u(\varphi); \varphi \in \cD((0,T) \times \mathbb{R}^{d})\}$
be the distribution solution of (\ref{heat-eq}). It is known that
\begin{equation}
\label{analogy} u(\varphi)=(G
* \dot F)(\varphi)= \dot F(\varphi
* \tilde G)
\end{equation}
where $G$ is the fundamental solution of the heat equation:
$$G(t,x)=\left\{
\begin{array}{ll} (4 \pi t)^{-d/2} \exp\left( -\frac{|x|^{2}}{4t}\right) & \mbox{if $t>0, x \in R^{d}$} \\
0 & \mbox{if $t \leq 0, x \in R^{d}$}
\end{array} \right.$$
(Since $G \in {\cO}_{C}'(\mathbb{R}^{d+1})$, we have $\varphi *
\tilde G \in {\cS}(\mathbb{R}^{d+1})$ for every $\varphi \in
{\cD}(\mathbb{R}^{d+1})$ and the convolution $G * \dot F$ is well
defined; see Theorem XI, Chapter VII, \cite{schwartz66}.)

\vspace{2mm}

Going back to our framework, we have the following lemma.

\begin{lemma}
\label{psi*Gtilde} If
\begin{equation}
\label{cond-mu} \int_{R^{d}} \frac{1}{1+|\xi|^{2}} \mu(d \xi)<
\infty,
\end{equation}
then: (a) $(G_{tx})^{\verb2~2} \in {\cP}_{+}^{(d)}$ for every $(t,x)
\in [0,T] \times \mathbb{R}^{d}$; (b) $\varphi * \tilde G \in
\cP_0^{(d)}$ for every $\varphi \in \cD((0,T) \times \bR^d)$.
\end{lemma}

\begin{proof} (a) Set $g_{t,x}=
(G_{tx})^{\verb2~2}$. Note that $g_{tx}$ is measurable and $\|
g_{tx} \|_{+}=\| g_{tx} \|_{0}$, since $g_{tx} \geq 0$. We have $\cF
g_{t,x}(s,\xi) = c_d 1_{t>s} \exp ( - \mathrm{i} \xi x- (t-s)
|\xi|^2)$, where $c_d$ is an appropriate constant. Hence
$$\| g_{tx} \|_{0}^{2}=
\int_0^T \int_{\bR^d} | \cF g_{t,x}(s,\xi) |^2 \, \mu (d \xi) \, ds
\leq N(T) \int_{\bR^d} \frac{1}{1+|\xi|^2} \, \mu ( d \xi) < \infty,
$$
where $N(T)$ is a constant depending on $T$.

(b) We apply Remark 4, \cite{dalang99} to the function $\psi=\varphi
* \tilde G$. For this we need to check that: (i) $t \mapsto {\cF}
\psi(t,\xi)$ is continuous $\forall \xi \in R^{d}$; (ii) there
exists a nonnegative function $k(t, \xi)$ which is square-integrable
with respect to $dt \times \mu(d \xi)$ such that $|{\cF}\psi(t,\xi)|
\leq k(t, \xi)$, $\forall t \in [0,T], \xi \in R^{d}$.

(i) is clearly satisfied, and (ii) will follow from (\ref{cond-mu})
once we prove that
\begin{equation} \label{boundedness}
|\cF \psi(t,\xi)| \le \frac{N}{1+|\xi|^{2}}:=k(t,\xi), \ \ \ \forall
t \in [0,T],\ \forall \xi \in \mathbb{R}^{d},
\end{equation}
where $N$ is a constant. Since $(1+|\xi|^{2})|\cF \psi(t,x)|=|{\cF
\phi}(t,\xi)| \leq \| \phi(t,\cdot) \|_{L_1(\mathbb{R}^d)}$ where
$\phi=(1-\Delta)\psi$, relation (\ref{boundedness}) follows if we
prove that $\| \phi(t,\cdot) \|_{L_1(\mathbb{R}^d)} \le N$. Note
that
$$
- \phi_t - \Delta \phi = (1- \Delta) \varphi \quad \text{in} \quad
(0,T) \times \mathbb{R}^d, \quad \phi(T,x) = 0,
$$
since $\psi$ is the unique solution of: $- \psi_t - \Delta \psi =
\varphi \ \text{in} \ (0,T) \times \mathbb{R}^d, \ \psi(T,x) = 0$.
Thus $\phi(s,y) = \int_s^T \int_{\mathbb{R}^d} G(t-s,x-y)
(1-\Delta)\varphi(t,x) \, dx \, dt$. From this we calculate $\|
\phi(s,\cdot) \|_{L_1(\mathbb{R}^d)}$, which turns out to be bounded
(note that $(1-\Delta)\varphi$ has a compact support in $(0,T)
\times \bR^d$).
\end{proof}

A consequence of Lemma \ref{psi*Gtilde}.(b) is the fact that
$M(\varphi
* \tilde G)$ is well-defined for every
$\varphi \in {\cD}((0,T) \times \mathbb{R}^{d})$. By analogy with
(\ref{analogy}) (and using a slight abuse of terminology), we
introduce the following definition:

\begin{definition}
{\rm The process $\{u(\varphi); \varphi \in {\cD}((0,T) \times
\mathbb{R}^{d})\}$ defined by
$$u(\varphi):  = M(\varphi * \tilde G)=
\int_{0}^{T} \int_{\mathbb{R}^{d}} \left( \int_{\mathbb{R}_{+}}
\int_{\mathbb{R}^{d}}\varphi(t+s,x+y)G(s,y)dyds \right) M(dt,dx)$$
is called the {\bf distribution-valued solution} of the stochastic
heat equation (\ref{heat-eq}), with vanishing initial conditions.}
\end{definition}

\begin{theorem}[\cite{dalang99}, \cite{dalang-mueller03}]
\label{dalang-existence-theorem} Let $\{u(\varphi); \varphi \in
{\cD}((0,T) \times \mathbb{R}^{d})\}$ be the distribution-valued
solution of the stochastic heat equation (\ref{heat-eq}). In order
that there exists a jointly measurable process $X=\{X(t,x);t \in
[0,T],x \in \mathbb{R}^{d}\}$ such that
$$u(\varphi)=\int_{0}^{T} \int_{\mathbb{R}^{d}} X(t,x)\varphi(t,x)
dxdt \ \ \ \forall \varphi \in {\cD}((0,T) \times \mathbb{R}^{d}) \
\ \ a.s.$$ it is necessary and sufficient that (\ref{cond-mu})
holds. In this case, $X$ is a modification of the process
$u=\{u(t,x);t \in [0,T],x \in \mathbb{R}^{d}\}$ defined by
\begin{equation}
\label{solution-process} u(t,x):=M((G_{tx})^{\verb2~2})=\int_{0}^{T}
\int_{\mathbb{R}^{d}} G(t-s,x-y)M(ds,dy).
\end{equation}
\end{theorem}

\begin{definition}
{\rm The process $u=\{u(t,x);t \in [0,T], x \in \mathbb{R}^{d}\}$
defined by (\ref{solution-process}) is called the {\bf process
solution} of the stochastic heat equation (\ref{heat-eq}), with
vanishing initial conditions.}
\end{definition}

\noindent Note that the process $u$ is a zero-mean Gaussian process.
In the present paper, we are interested in examining the germ Markov
property of the process $u$.

\subsection{The Gaussian space}

The Gaussian space $H^{u}$ of the process $u$ is defined as the
closed linear subspace of $L_{2}(\Omega)$, generated by the
variables $\{u(t,x), t \in [0,T], x \in \bR^{d}\}$.
The next result shows that this space coincides with the space
$H^{F}$.

\begin{lemma}
\label{Hu=HF} We have $H^{u}=H^{F}$.
\end{lemma}

\begin{proof} a) First, we prove that $H^{u} \subseteq H^{F}$. Recall
that $u(t,x)=M((G_{tx})^{\verb2~2})$ and $(G_{tx})^{\verb2~2} \in
{\cP}_{+}^{(d)} \subset {\cP}_{0}^{(d)}$, by Lemma
\ref{psi*Gtilde}.(a). Since ${\cD}((0,T) \times \mathbb{R}^{d})$ is
dense in ${\cP}_{0}^{(d)}$ with respect to $\| \cdot \|_{0}$, there
exists a sequence $\{\varphi_{n}\}_{n \geq 1} \subseteq {\cD}((0,T)
\times \mathbb{R}^{d})$ such that
$\|\varphi_{n}-(G_{tx})^{\verb2~2}\|_{0} \to 0$. Hence
$\bE(M(\varphi_{n})-u(t,x))^{2} \to 0$. But $M(\varphi_{n}) \in
H^{F}$ for all $n$ and therefore $u(t,x) \in H^{F}$.

b) Let $H_{*}^{u}$ be the closed linear subspace of $L_{2}(\Omega)$,
generated by the variables $\{u(\eta), \eta \in \cD((0,T) \times
\bR^{d})\}$. To prove that $H^F \subseteq H_{*}^{u}$, let $\varphi
\in \cD((0,T) \times \mathbb{R}^{d})$ be arbitrary and
$\eta=-\varphi_{t}-\Delta \varphi \in \cD((0,T) \times
\mathbb{R}^{d})$. Then $\varphi=\eta
* \tilde G$ and $F(\varphi)=M(\varphi)=M(\eta * \tilde G)=u(\eta) \in H_{*}^{u}$.

c) Finally, we prove that $H_{*}^{u} \subseteq H^{u}$. Let $\eta \in
{\cD}((0,T) \times \mathbb{R}^{d})$ be arbitrary. Note that
$u(\eta)=M(\varphi)$ where $\varphi=\eta* \tilde G$. Let $K$ be a
compact set such that $\text{supp} \, \eta \subset [0,T] \times K$.

For each $n \geq 1$, let $\{ Q_m^{(n)}, m \in \bZ \}$ be a partition
of $(0,T) \times \bR^d$ such that each $Q_m^{(n)}=R_{m}^{(n)} \times
S_{m}^{(n)}$, where $R_{m}^{(n)}$ is an interval in $[0,T]$ and
$S_{m}^{(n)}$ is a cube in $\bR^d$. Suppose that $\max_{m \in \bZ}
|Q_m^{(n)} | \to 0$ as $n \to \infty$. For each $Q_m^{(n)}$, we
choose $(t_m^{(n)}, x_m^{(n)}) \in Q_m^{(n)}$ such that $t_m^{(n)}
\ge t$ for all $t \in R_m^{(n)}$. We consider the Riemann sum:
$$\varphi_n(s,y) = \sum_{m \in I_n} |Q_m^{(n)}|  G(t_m^{(n)}-s,x_m^{(n)}-y)
\eta(t_m^{(n)},x_m^{(n)}),$$ where $I_n=\{m \in \bZ ; t_m^{(n)} > s,
x_m^{(n)} \in K\}$.
Clearly $\varphi_n(s,y) \rightarrow \varphi (s,y)$ for every
$(s,y)$. We claim that
\begin{equation}
\label{conv-norm0} \|\varphi_{n}-\varphi \|_0 \rightarrow 0.
\end{equation}

\noindent From here, it follows that
$\bE(M(\varphi_n)-M(\varphi))^{2} \to 0$. This concludes that proof,
since $\varphi_{n}=\sum_{m \in I_n}a_m^{(n)}
(G_{t_m^{(n)},x_m^{(n)}})^{\verb2~2}$ and hence
$M(\varphi_n)=\sum_{m \in I_n}a_m^{(n)} u(t_m^{(n)},x_m^{(n)}) \in
H^{u}$, where $a_{m}^{(n)}=|Q_m^{(n)}| \ \eta(t_m^{(n)},x_m^{(n)})$.
The proof of (\ref{conv-norm0}) is given in Appendix A.
\end{proof}

\subsection{The RKHS}

Let ${\cH}^{u}=\{h_{Y}; h_{Y}(t,x)=\bE (Yu(t,x)),Y \in H^{u}\}$ be
the RKHS of the process $u$, endowed with the inner product $\langle
h_{Y},h_{Z} \rangle_{{\cH}^{u}}:=\bE(YZ), \ Y,Z \in H^{u}$. Note
that any function $h \in \cH^{u}$ is continuous on $[0,T] \times
\mathbb{R}^{d}$ and satisfies $h(0,\cdot)=0$. Moreover, $\cH^u$ is
the closure of $\{R((t,x), \cdot \ );(t,x) \in [0,T] \times \bR^d\}$
with respect to $\| \cdot \|_{\cH^u}$, where
$R((t,x),(s,y))=\bE(u(t,x)u(s,y))$.

By Lemma \ref{Hu=HF} and (\ref{M-HF}), it follows that
$${\cH}^{u}=\{h(t,x)=\bE (M(\varphi)u(t,x)); \varphi \in
{\cP}_{0}^{(d)}\}.$$ Moreover, if $h(t,x)=\bE (M(\varphi)u(t,x))$,
$g(t,x)=\bE (M(\eta)u(t,x))$ with $\varphi,\eta \in \cP_0^{(d)}$,
then $\langle h,g \rangle_{{\cH}^{u}}=\langle \varphi,\eta
\rangle_{0}$.

Let $g_{t,x}= (G_{tx})^{\verb2~2}$. Using the fact that
$u(t,x)=M(g_{tx})$ and (\ref{product0-produxt0x}), we get:
\begin{equation}
\label{formula-h} h(t,x) = \bE (M(\varphi) u(t,x))=\langle \varphi,
g_{tx} \rangle_{0}=
\int_0^t \int_{\bR^d} \cF \varphi (s,\xi) \overline{ \cF
g_{t,x}(s,\xi)} \, \mu(d\xi) \, ds.
\end{equation}

\begin{lemma}
\label{key-calculation} Let  $h(t,x) = \bE (M(\varphi) u(t,x)),
\varphi \in \cP_0^{(d)}$. For $\eta \in \cD((0,T) \times \bR^{d})$,
let $\phi$ be a solution to the equation $-\phi_t-\Delta \phi=\eta$
in $(0,T) \times \bR^d$, $\phi(T,\cdot)=0$. Then
$$
\langle h, \eta \rangle_{L_{2}((0,T) \times \bR^d)} =\langle
\varphi, \phi \rangle_{0}.
$$
 In particular, for any
$\phi \in \cD((0,T) \times \bR^d)$, we have
$$\langle h, -\phi_{t}-\Delta \phi \rangle_{L_{2}((0,T) \times \bR^{d})}=\langle
\varphi,\phi \rangle_{0}.$$
\end{lemma}

\begin{proof}
Using (\ref{formula-h}) and applying Fubini's theorem (since
$\varphi \in \cP_{0}^{(d)}$, $g_{tx} \in \cP_{+}^{(d)}$, and $\eta$
has compact support), we get
\begin{multline}\label{interm-calcul1}
\langle h, \eta \rangle_{L_{2}((0,T) \times \bR^d)}
= \int_0^T \langle h(s,\cdot), \eta(s,\cdot) \rangle_{L_2(\bR^d)} \, ds\\
= \int_0^T \int_{\bR^d} \cF \varphi (r,\xi) \int_r^T \int_{\bR^d}
\overline{\cF g_{s,x}(r,\xi) \, \eta(s,x)} \, dx \, ds \, \mu(d\xi)
\, dr.
\end{multline}


Note that 
for every $r \in (0,T)$ and $y \in \bR^d$
$$
\phi(r,y) = \int_r^T \int_{\bR^d} G(s-r,x-y) \eta(s,x) \, dx \,ds.
$$
Hence
\begin{equation}
\label{interm-calcul2} \cF \phi(r,\xi)= \int_r^T \int_{\bR^d} \cF
g_{s,x}(r,\xi) \eta(s,x) \, dx \, ds
\end{equation}

\noindent From (\ref{interm-calcul1}) and (\ref{interm-calcul2}), we
conclude that
$$
\langle h, \eta \rangle_{L_{2}((0,T) \times \bR^d)}
= \langle \varphi, \phi \rangle_0
$$
(recall that
$\phi \in \cP_0^{(d)}$ by Lemma \ref{psi*Gtilde}). This finishes the
proof.
\end{proof}

\section{The Germ Markov Property}
\subsection{The Definition}

Let $S \subseteq [0,T] \times \mathbb{R}^{d}$ be an arbitrary set.
Let ${\cF}_{S}^{u}$ be the $\sigma$-field generated by the variables
$\{u(t,x);(t,x) \in S\}$, $K^{u}_{S}$ be the closed linear subspace
of $L_{2}(\Omega)$ generated by the variables $\{u(t,x);(t,x) \in
S\}$, and ${\cK}_{S}^{u}$ be the closed linear subspace of
${\cH}^{u}$ generated by the functions $\{R((t,x), \cdot \ ); (t,x)
\in S\}$. Let
$${\cG}_{S}^{u}
=\bigcap_{O \ {\rm open}; \ O \supset S}{\cF}_{O}^{u}, \ \ \
H_{S}^{u} =\bigcap_{O \ {\rm open}; \ O \supset S}{K}_{O}^{u}, \ \ \
{\cH}_{S}^{u} =\bigcap_{O \ {\rm open}; \ O \supset
S}{\cK}_{O}^{u}.$$

\begin{definition}
{\rm The process $u=\{u(t,x); (t,x) \in [0,T] \times \bR^d\}$ is
{\bf locally germ Markov} if for every relatively compact open set
$A \subset [0,T] \times \mathbb{R}^{d}$, ${\cG}_{\overline{A}}^{u}$
and ${\cG}_{\overline{A^{c}}}^{u}$ are conditionally independent
given ${\cG}_{\partial A}^{u}$, where $\partial A=\overline A \cap
\overline{A^{c}}$. }
\end{definition}







Based on the fact that ${\cG}_{S}^{u}=\sigma(H_{S}^{u})$,
 one can prove that $u$ is locally
germ Markov if and only if for every open set $A$,
$H^{u}=H_{\overline{A}}^{u} \oplus (H_{\overline{A^{c}}}^{u} \ominus
H_{\partial A}^{u})$ (see Lemma 1.3, \cite{kunsch79}), or
equivalently, using the isometry between $H^{u}$ and $\cH^{u}$
\begin{equation}
\label{direct-product} {\cH}^{u}={\cH}_{\overline{A}}^{u} \oplus
({\cH}_{\overline{A^{c}}}^{u} \ominus {\cH}_{\partial A}^{u}).
\end{equation}
(Here $\oplus, \ominus$ denote the usual operations on Hilbert
subspaces: if $H$ is a Hilbert space, $S$ is a closed subspace and
$S^{\perp}$ is its orthogonal complement, then we write $H=S \oplus
S^{\perp}$ and $S^{\perp}=H \ominus S$.)

On the other hand, $h(t,x)=\langle h,R((t,x), \cdot \ )
\rangle_{{\cH}^{u}}$ for every $h \in {\cH}^{u}$. Hence
\begin{equation}
\label{support-orthog} {\rm supp} \ h \subseteq B^{c} \ \ \ \mbox{if
and only if} \ \ \ h \in ({\cH}_{B}^{u})^{\perp}.
\end{equation}

\noindent Based on (\ref{direct-product}) and
(\ref{support-orthog}), we have the following fundamental result.

\begin{theorem}[Theorem 5.1, \cite{kunsch79}]
\label{kunsch-theorem} The Gaussian process $u$ is locally germ
Markov if and only if the following two conditions hold:

(i) If $h,g \in {\cH}^{u}$ are such that $({\rm supp} \ h) \cap
({\rm supp} \ g)=\emptyset$ and ${\rm supp} \ h$ is compact,
 then $<h,g>_{{\cH}^{u}}=0$.

(ii)  If $\zeta=h+g \in {\cH}^{u}$, where $h,g$ are such that $({\rm
supp} \ h) \cap ({\rm supp} \ g)=\emptyset$ and ${\rm supp} \ h$ is
compact, then $h,g \in {\cH}^{u}$.
\end{theorem}

In the next two subsections, we will suppose that $f$ is the Bessel
kernel of order $\alpha> 0$ (Example \ref{bessel}), respectively the
Riesz kernel of order $0<\alpha<d$ (Example \ref{riesz}). In both
cases we must have $\alpha>d-2$, in order to have condition
(\ref{cond-mu}) satisfied. Our goal is to prove that conditions (i)
and (ii) of Theorem \ref{kunsch-theorem} hold. For this, we will
assume that $\alpha=2k, k \in \bN_+$ in the case of the the Bessel
kernel, respectively $\alpha=4k, k \in \bN_+$ in the case of the
Riesz kernel (and $d=4k+1$, since $d-2<\alpha<d$).

\subsection{The Case of the Bessel Kernel}

In this subsection we will assume that $f$ is the Bessel kernel,
i.e. $f=B_{\alpha}$ with $\alpha>\max\{0,d-2\}$.

 In this case,
$\cP_{0,x}^{(d)}
=H_{2}^{-\alpha/2}(\bR^{d})$, where $$H_{2}^{\gamma}(\bR^d)
=\{\varphi \in \cS'(\bR^d); \cF \varphi \ \mbox{is a
function},\|\varphi \|_{\gamma,2}^2=\int_{\bR^d}|\cF \varphi
(\xi)|^{2} (1+|\xi|^{2})^{\gamma}d\xi<\infty\}$$ denotes the
fractional Sobolev space of index $\gamma \in \bR$ (see e.g.
\cite{folland95}, p.191-192).

By Remark \ref{P0d-in-L2}, $\cP_0^{(d)} \subset L_2((0,T),
H_2^{-\alpha/2}(\bR^d))$.
On the other hand, we have the denseness of $\cE_0^{(d)}$ in
$L_2((0,T),H_2^{-\alpha/2}(\bR^d))$ (for a proof, one may see the
proof of Theorem 3.10 in \cite{krylov99}). Thus $\cP_0^{(d)} =
L_2((0,T),H_2^{-\alpha/2}(\bR^d))$ and
$$\| \varphi \|_{0}=\| \varphi
\|_{L_{2}((0,T),H_{2}^{-\alpha/2}(\bR^{d}))}, \ \ \ \forall \varphi
\in \cP_{0}^{(d)}.$$
\noindent For each $t \in [0,T]$, let $\varphi_{1}(t,
\cdot):=(1-\Delta)^{-\alpha/2}\varphi (t, \cdot)$ and note that the
map $\varphi \mapsto \varphi_{1}$ is an isometry between $L_2((0,T),
H_2^{-\alpha/2}(\bR^d))$ and $L_2((0,T), H_2^{\alpha/2}(\bR^d))$.

Since
${\cF} \varphi_{1}(t,\xi)=(1+|\xi|^{2})^{-\alpha/2} \cF
\varphi(t,\xi)$, it is not difficult to see that: $\forall \varphi
\in L_2((0,T), H_2^{-\alpha/2}(\bR^d))$
\begin{equation}
\label{product-varphi1} \langle \varphi_{1}, \phi
\rangle_{L_{2}((0,T) \times \bR^{d})}=\langle \varphi, \phi
\rangle_{0}, \ \ \ \forall \phi \in {\cD}((0,T) \times \bR^{d}).
\end{equation}


To investigate the relationship between $h$ and $\varphi_1$,
we need some general results from the $L_p$-theory of parabolic
equations.

Recall that $H_2^{0}(\bR^d)=L_2(\bR^d)$, and $H_2^{\gamma}(\bR^d)
\subset H_2^{\gamma'}(\bR^d)$ if $\gamma>\gamma'$. For every
$\gamma,\beta \in \bR$, $(1-\Delta)^{\gamma/2}$ is a unitary
isomorphism between $H_2^{\beta}(\bR^d)$ and
$H_{2}^{\gamma-\beta}(\bR^d)$. In particular,
$(1-\Delta)^{\gamma/2}$ is a unitary isomorphism between
$H_2^{\gamma}(\bR^d)$ and $L_{2}(\bR^d)$. For every $v \in
H_2^{\gamma}(\bR^d)$, $\phi \in \cD(\bR^d)$, we denote
$$(v,\phi)=\int_{\bR^d}
[(1-\Delta)^{\gamma/2}v](x)\overline{[(1-\Delta)^{-\gamma/2}\phi](x)}dx.$$
Note that, if $v \in H_2^{\gamma}(\bR^d)$ and $\gamma \geq 0$, then
$(v,\phi)=\langle v,\phi \rangle_{L_{2}(\bR^d)}$ for all $\phi \in
\cD(\bR^d)$.

\begin{definition}
{\rm If $t \mapsto v(t, \cdot)$ is a function from $[0,T]$ to
$H_2^{\gamma}(\bR^d)$, with $\gamma \in \bR$, we say that $v$ is a
{\bf solution} of
$$d v = ( \Delta v + g ) d t \quad \text{in} \quad (0,T) \times
\bR^d, \quad v(0,\cdot) = 0$$ if for any $t \in (0,T)$ and for any
$\psi \in \cD(\bR^d)$, we have}
\begin{equation}\label{Krylov-eq}
(v(t,\cdot), \psi) = \int_0^t (v(s,\cdot), \Delta \psi)ds +\int_0^t
(g(s,\cdot), \psi) ds.
\end{equation}
We write $v \in \cH_{2,0}^{\gamma}(T)$ if $v_{xx} \in L_2((0,T),
H_2^{\gamma-2}(\bR^d))$, $v(0,\cdot) = 0$, and there exists $g \in
L_2((0,T), H_2^{\gamma-2}(\bR^d))$ satisfying \eqref{Krylov-eq}. By
$\|v\|_{\cH_{2,0}^{\gamma}(T)}$ we mean
$$
\|v\|_{\cH_{2,0}^{\gamma}(T)} = \|v_{xx}\|_{L_2((0,T),
H_2^{\gamma-2}(\bR^d))} + \|g\|_{L_2((0,T), H_2^{\gamma-2}(\bR^d))}.
$$

\end{definition}

\begin{remark}
{\rm It is known that (\ref{Krylov-eq}) implies that: $\forall t \in
(0,T)$, $\forall \phi \in \cD((0,T) \times \bR^d)$
$$(v(t,\cdot),
\phi(t,\cdot))=\int_0^t (v(s,\cdot), (\partial/\partial t+\Delta)
\phi(s,\cdot))ds +\int_0^t (g(s,\cdot), \phi(s,\cdot)) ds$$ (see
e.g. Proposition 10, \cite{ferrante-sanzsole06} for a stochastic
version of this result). In particular, (\ref{Krylov-eq}) implies
that: $\forall \phi \in \cD((0,T) \times \bR^d)$
\begin{equation}
\label{Krylov-eq2} \int_0^T (v(s,\cdot), (-\partial/\partial
t-\Delta) \phi(s,\cdot))ds =\int_0^T (g(s,\cdot), \phi(s,\cdot)) ds
\end{equation}
(by taking $t=T$ and using the fact that $\phi(T, \cdot)=0$).}
\end{remark}

\begin{theorem}[\cite{krylov99}, \cite{krylov01}]
\label{theoremA} Given $g \in L_2((0,T), H_2^{\gamma}(\bR^d))$,
$\gamma \in \bR$, there exists a unique solution
$v \in \cH_{2,0}^{\gamma+2}(T)$ to the equation
$$d v = ( \Delta v + g ) d t \quad \text{in} \quad (0,T) \times \bR^d,
\quad v(0,\cdot) = 0.$$
Moreover, there exists a constant $N$ (independent of $v$) such that
$$
\| v \|_{\cH_{2,0}^{\gamma+2}(T)} \le N \| g\|_{L_2((0,T),
L_2(\bR^d))}.
$$
\end{theorem}

We now return to our framework.

\begin{theorem}\label{h-unique-alpha+2}
Let $h(t,x) = \bE (M(\varphi) u(t,x))$, where $u$ is defined in
\eqref{solution-process} and $\varphi \in \cP_0^{(d)}$. Set
$\varphi_1=(1-\Delta)^{-\alpha/2}\varphi$. Then $h$ is the unique
solution in
$\cH_{2,0}^{\alpha/2+2}(T)$ to the equation
\begin{equation}
\label{Krylov-equation} d h = (\Delta h + \varphi_1) \, dt \ \ \ in
\ (0,T) \times \bR^{d}, \ \ \
 \quad h(0,\cdot) = 0.
 \end{equation}
\end{theorem}

\begin{proof}
Recall that $\varphi_1 \in L_{2}((0,T), H_2^{\alpha/2}(\bR^d))$.
Thus by Theorem \ref{theoremA} there exists a unique solution $v \in
\cH_{2,0}^{\alpha/2+2}(T)$ to the equation \eqref{Krylov-equation}.

We are now proving that $h=v$. This will follow once we prove that
$$
\langle h,\eta \rangle_{L_2((0,T) \times \bR^d)} = \langle v, \eta
\rangle_{L_2((0,T)\times\bR^d)}, \quad \forall{\eta \in \cD((0,T)
\times \bR^d)}.
$$
Let $\eta \in \cD((0,T) \times \bR^d)$ be arbitrary and $\phi$ be
the unique solution of
$$-\phi_t-\Delta \phi=\eta \quad \mbox{in} \ (0,T) \times \bR^{d}, \quad \phi(T,\cdot) = 0.$$
By Lemma \ref{key-calculation} and (\ref{product-varphi1})
$$
\langle h,\eta \rangle_{L_2((0,T) \times \bR^d)} =\langle \varphi,
\phi \rangle_{0} =\langle \varphi_1, \phi
\rangle_{L_2((0,T)\times\bR^d)}.
$$
On the other hand, by \eqref{Krylov-eq2} we have
$$
\langle v, \eta \rangle_{L_2((0,T)\times\bR^d)} = \langle \varphi_1,
\phi \rangle_{L_2((0,T) \times \bR^d)}
$$
This concludes the proof.
\end{proof}

\begin{corollary}\label{corollary20070928}
If $f$ is the Bessel kernel of order $\alpha$, then
$$
\cH^u = \cH_{2,0}^{\alpha/2+2}(T)
$$
and the norms in the two spaces are equivalent.
\end{corollary}

\begin{proof}
From the argument at the beginning of subsection 3.3, for every $h
\in \cH^u$, we have $h(t,x) = \bE M(\varphi)u(t,x)$, where $u$ is
defined in \eqref{solution-process} and $\varphi \in \cP_0^{(d)}$.
Then by Theorem \ref{h-unique-alpha+2}, $h \in
\cH_{2,0}^{\alpha/2+2}(T)$.

For $v \in \cH_{2,0}^{\alpha/2+2}(T)$, there is a $\varphi_1 \in
L_2((0,T), H_2^{\alpha/2}(\bR^d))$ satisfying \eqref{Krylov-eq} with
$\varphi_1$ in place of $g$. Then
$$
\varphi:= (1-\Delta)^{\alpha/2}\varphi_1 \in
L_2((0,T),H_2^{-\alpha/2}(\bR^d)) = \cP_0^{(d)}.
$$
Set $h(t,x) = \bE M(\varphi)u(t,x) \in \cH^u$, then by Theorem
\ref{h-unique-alpha+2} it follows that $v=h$, so $v \in \cH^u$.

Now notice that
$$
\| h \|_{\cH^u} = \| \varphi \|_{0} = \| \varphi \|_{L_2((0,T),
H_2^{-\alpha/2}(\bR^d))} = \| (1-\Delta)^{-\alpha/2} \varphi
\|_{L_2((0,T), H_2^{\alpha/2}(\bR^d))}
$$
$$
= \| \varphi_1 \|_{L_2((0,T), H_2^{\alpha/2}(\bR^d))} \simeq \| h
\|_{\cH_{2,0}^{\alpha/2+2}(T)},
$$
where $\simeq$ indicates the equivalence of the norms, which follows
from the definition of the norm of $\cH_{2,0}^{\alpha/2+2}(T)$ and
the estimate in Theorem \ref{theoremA}. This finishes the proof.
\end{proof}

\begin{remark}\label{remark20070928}
{\rm Under the conditions of Theorem \ref{h-unique-alpha+2}, we can
also say that $h$ is the unique solution in $W_{2}^{1,2}((0,T)
\times \bR^{d})$ of:
$$h_{t} = \Delta h + \varphi_1 \ \ \  {\rm in} \ \ \ (0,T) \times
\bR^{d}, \quad h(0,\cdot) = 0.$$ Here $W_2^{1,2}((0,T) \times
\bR^d)$ is the space of all measurable functions $v:(0,T) \times
\bR^d \rightarrow \bR$, such that the weak derivatives $v_t$,
$v_{x_i}$, $v_{x_i x_j}$ exist and are in $L_2((0,T) \times
\bR^d)$.}
\end{remark}

We are now ready to prove the main result of this subsection.

\begin{theorem}
\label{main-theorem-Bessel} Suppose that $f$ is the Bessel kernel of
order $\alpha=2k,k \in \bN_{+}$ such that $\alpha>d-2$. Then the
process solution $u$ of the stochastic heat equation (\ref{heat-eq})
with vanishing initial conditions is locally germ Markov.
\end{theorem}

\begin{proof}
We need to verify conditions (i) and (ii) of Theorem
\ref{kunsch-theorem}.

We first verify condition (i). Let $h,g \in {\cH}^{u}$ are such that
$({\rm supp} \ h) \cap ({\rm supp} \ g)=\emptyset$ and ${\rm supp} \
h$ is compact. We have to prove that $\langle h,g
\rangle_{\cH^u}=0$. We know that there exist $\varphi, \eta \in
\cP_0^{(d)}$ such that $h(t,x) = \bE (M(\varphi) u(t,x))$ and
$g(t,x) = \bE (M(\eta) u(t,x))$. Then by Theorem
\ref{h-unique-alpha+2} and Corollary \ref{corollary20070928} (also
recall the definition of the norm of $\cH_{2,0}^{\gamma}(T)$),
$$
\langle h, g \rangle_{\cH^u} = \langle h, g
\rangle_{\cH_{2,0}^{k+2}(T)} = \langle h_{xx}, g_{xx}
\rangle_{L_2((0,T), H_2^k(\bR^d))} + \langle \varphi_1, \eta_1
\rangle_{L_2((0,T), H_2^k(\bR^d))},
$$
where $\varphi_1 = (1-\Delta)^k\varphi$ and $\eta_1 =
(1-\Delta)^k\eta$. Note that (see also Remark \ref{remark20070928})
$$
\text{supp}\, h_{xx} \subset \text{supp}\, h, \quad \text{supp}\,
\varphi_1 \subset \text{supp}\, h.
$$
Similar inclusions hold for $g$ and $\eta_1$. Thus it is clear that
$$
\langle h_{xx}, g_{xx} \rangle_{L_2((0,T), H_2^k(\bR^d))} = \langle
\varphi_1, \eta_1 \rangle_{L_2((0,T), H_2^k(\bR^d))} = 0,
$$
from which we arrive at $\langle h, g \rangle_{\cH^u} = 0$. This
finishes the proof of the condition (i).

To prove the condition (ii), let $\zeta=h+g \in {\cH}^{u}$, where
$h$ and $g$ are such that $({\rm supp} \ h) \cap ({\rm supp} \
g)=\emptyset$ and ${\rm supp} \ h$ is compact. We have to prove that
$h \in \cH^u$.

Let $\chi$ be an infinitely differentiable function such that
$\chi=1$ on ${\rm supp} \ h$ and $\chi=0$ on an open set containing
${\rm supp} \ g$. By Corollary \ref{corollary20070928}, it follows
that $\zeta \in \cH_{2,0}^{k+2}(T)$, and hence $h = \chi \zeta \in
\cH_{2,0}^{k+2}(T)= \cH^u$ . The theorem is proved.
\end{proof}

\subsection{The Case of the Riesz Kernel}

In this subsection, we will assume that $f$ is the Riesz kernel,
i.e. $f=R_{\alpha}$ with $\max\{d-2,0\}< \alpha = 4k<d$, $k \in
\bN_+$.

According to \cite{samko76}, for any $0<\beta<d/2$ we can define the
Riesz potential
$$
I^{\beta}\varphi(y):=\varphi
*R_{\beta}(y)=\frac{1}{\gamma_{n}(\beta)}
\int_{\bR^d}\frac{\varphi(y)}{|x-y|^{d-\beta}}dy \quad \mbox{for
all} \ \varphi \in L_{2}(\bR^{d}).$$ Let $q=q_{d,\beta}>2$ be such
that $1/q=1/2-\beta/d$.

The space $I^{\beta}(L_{2}(R^d))$ of all Riesz potentials has the
following properties (see \cite{samko76}):
\begin{enumerate}
\item $I^{\beta}(L_{2}(R^d)) \subset L_{q}(\bR^d)$ and there exists a
constant $N>0$ such that $$\|I^{\beta}\varphi \|_{L_{q}(\bR^d)} \leq
N \|\varphi \|_{L_{2}(\bR^d)}, \quad \mbox{for all} \ \varphi \in
L_{2}(\bR^d).$$

\item For every $f \in I^{\beta}(L_{2}(R^d))$, we define the Riesz
derivative $\bD^{\beta}f$ as \linebreak $\lim_{\varepsilon \to
0}\bD_{\varepsilon}^{\beta}f$ in $L_{2}(\bR^d)$, where
$$(\bD_{\varepsilon}^{\beta}f)(x)=
\frac{1}{c_{d,l}(\beta)}\int_{|y|>\varepsilon}
\frac{(\Delta_y^{l}f)(x)}{|y|^{d+\beta}}dy \quad (\mbox{is
independent of} \ l),$$ $\Delta_{y}^{l}f=(I-\tau_y)^{l}f$ and
$(\tau_y f)(x)=f(x-y)$. Then
$$
\cF (\bD^{\beta}f) (\xi)={\cF}f(\xi)|\xi|^{\beta}, \ \ \ \forall f
\in \cS(\bR^d).
$$

\item $I^{\beta}(L_{2}(R^d))=\{f \in L_{q}(\bR^d); \bD^{\beta}f \
\mbox{exists and is in} \ L_{2}(\bR^d)\}$.

\item $\bD^{\beta}$ is the left inverse of $I^{\beta}$ in
$L_{2}(\bR^d)$, i.e. $\bD^{\beta}(I^{\beta}\varphi)=\varphi$ for all
$\varphi \in
L_{2}(\bR^d)$. 

\item $\cD(\bR^d)$ is dense in $I^{\beta}(L_{2}(R^d))$ with respect to
 the norm $\|f\|_{L_{q}(\bR^d)}+\linebreak \|\bD^{\beta}f
 \|_{L_{2}(\bR^d)}$.


\end{enumerate}






In what follows, we will use these properties with
$\beta=\alpha/2=2k$, where $d-2<4k<d$. We let $q=q_{d,4k}>2$ be such
that $1/q=1/2-2k/d$.






\begin{proposition}
\label{Fourier-varphi1-Riesz} Let $2 < d/(2k)$ and $1/q = 1/2 -
2k/d$. If $f \in I^{2k}(L_2(\bR^d))$, where $k \in \bN_+$, then
$$
\cF \left( \bD^{2k} f \right) = |\xi|^{2k} \cF f,
$$
where $\cF f$ is an element in $\cS'(\bR^d)$.
\end{proposition}

\begin{proof}
First of all, $|\xi|^{2k} \cF f$ is well-defined as an element of
$\cS'$ since $|\xi|^{2k} = ( \xi_1^2 + \cdots + \xi_d^2 )^{k}$ is
infinitely differentiable.

Due to the fact that $\cD(\bR^d)$ is dense in $I^{2k}(L_2(\bR^d))$,
there exists a sequence $\{ f_n \} \subset \cD(\bR^d)$ such that
$$
\| f_n - f \|_{L_q(\bR^d)} + \| \bD^{2k} f_n - \bD^{2k} f
\|_{L_2(\bR^d)} \to 0.
$$
Note that $\cF \left( \bD^{2k} f_n \right) = |\xi|^{2k} \cF f_n \in
\cS(\bR^d)$. Also note that
$$
\left( \cF \left( \bD^{2k} f_n \right), \phi \right) = \left(
|\xi|^{2k} \cF f_n, \phi \right) = \left( \cF f_n, |\xi|^{2k} \phi
\right)
$$
$$
= \left( f_n, \cF^{-1} \left(|\xi|^{2k} \phi \right) \right) \to
\left( f, \cF^{-1} \left(|\xi|^{2k} \phi \right) \right) = \left(
|\xi|^{2k} \cF f, \phi \right),
$$
where the convergence is possible because $f_n \to f$ in
$L_q(\bR^d)$. This along with the fact that $\bD^{2k} f_n \to
\bD^{2k} f$ in $L_2(\bR^d)$ implies that $$ \cF \left( \bD^{2k} f
\right) = |\xi|^{2k} \cF f.
$$
\end{proof}

\begin{corollary} \label{claim_2006_12_19}
Let $2 < d/(2k)$ and $1/q = 1/2 - 2k/d$. If $f \in
I^{2k}(L_2(\bR^d))$, where $k \in \bN_+$, then
$$
\bD^{2k} f = (-\Delta)^k f.
$$
\end{corollary}

\begin{proof} Note that
$$
\left( \cF \left( \bD^{2k} f \right), \phi \right) = \left(
|\xi|^{2k} \cF f, \phi \right) = \left( \cF f, |\xi|^{2k} \phi
\right).
$$
On the other hand,
$$
\left( \cF \left( (-\Delta)^{k} f \right), \phi \right) = \left(
(-\Delta)^{k} f, \cF^{-1}\phi \right) = \left( f, (-\Delta)^{k}
\cF^{-1}\phi \right)
$$
$$
= \left( f, \cF^{-1} \left( |\xi|^{2k} \phi \right) \right) = \left(
\cF f, |\xi|^{2k} \phi \right).
$$

\end{proof}

\begin{lemma}\label{lemma20070927_01}
Let $2 < d/2k$. There exists a linear operator $J:\cP_0^{(d)} \to
L_2((0,T) \times \bR^d)$ such that
 $J$ is one-to-one and onto, and satisfies:
$$
J \varphi (t,x) = I^{2k} \left(\varphi(t,\cdot) \right) (x)
$$
for any $\varphi \in \cD((0,T) \times \bR^d)$.
\end{lemma}

\begin{proof}
For all $\varphi \in \cD((0,T) \times \bR^d)$, set
$$
J(\varphi) = I^{2k} \left( \varphi(t,\cdot) \right)(x) = c_{d,k}
\int_{\bR^d} \frac{\varphi(t,y)}{|x-y|^{d-2k}} \, dy,
$$
where $c_{d,k}$ is an appropriate constant. Denote $\varphi_0 :=
J(\varphi)$. We see that $\varphi_0(t,x)$ is measurable in $(t,x)
\in (0,T) \times \bR^d$. Since $\varphi(t,\cdot) \in \cS(\bR^d)$ for
each fixed $t \in (0,T)$, we have
$$
\cF \varphi_0(t,\cdot) (\xi) = \cF \varphi(t,\xi) |\xi|^{-2k}
$$
in the sense that
$$
\int_{\bR^d} \varphi_0(t,x) \overline{\phi(x)} \, dx = \int_{\bR^d}
\cF \varphi(t,\xi) |\xi|^{-2k} \overline{\cF \phi(\xi)} \, d \xi,
$$
where $\phi \in \cS(\bR^d)$. Also notice that, for a.e. $t \in
(0,T)$,
$$
\cF \varphi(t,\xi) |\xi|^{-2k} \in L_2(\bR^d)
$$
because $\varphi \in \cP_0^{(d)}$ (recall that $\alpha = 4k$), i.e.,
$$
\int_0^T\int_{\bR^d} |\cF \varphi(t,\xi) |\xi|^{-2k}|^2 \, d \xi \,
dt < \infty.
$$
Thus, for a.e. $t \in (0,T)$, as a function of $x \in \bR^d$,
$\varphi_0(t,x) \in L_2(\bR^d)$ and
$$
\cF \varphi_0 (t,\cdot)(\xi) = \cF \varphi(t,\xi) |\xi|^{-2k}.
$$
Thus
$$
\|J(\varphi)\|_{L_2((0,T) \times \bR^d)}^2 = \int_0^T \int_{\bR^d} |
\varphi_0(t,x) |^2 \, dx \, dt
$$
$$
= \int_0^T \int_{\bR^d} | \cF \varphi(t,\xi)|^2 |\xi|^{-4k} \, d \xi
\, dt = \| \varphi \|_0^2.
$$
From this and the fact that $\cD((0,T) \times \bR^d)$ is dense in
$\cP_0^{(d)}$, we extend $J$ to all elements in $\cP_0^{(d)}$.

It is easy to see that $J$ is linear as well as one-to-one. To prove
the fact that $J$ is onto, take $\varphi_0 \in L_2((0,T) \times
\bR^d)$. Also take a sequence $\{{\varphi_0}_n\} \subset \cD((0,T)
\times \bR^d)$ such that ${\varphi_0}_n \to \varphi_0$ in $L_2((0,T)
\times \bR^d)$. Let $\varphi_n = (-\Delta)^{k} {\varphi_0}_n$. Then
$$
\cF \varphi_n (t, \xi) = |\xi|^{2k} \cF {\varphi_0}_n (t, \xi)
$$
and
$$
\int_0^T \int_{\bR^d} |\cF \varphi_n(t,\xi)|^2 |\xi|^{-4k} \, d \xi
\, dt = \int_0^T \int_{\bR^d} |\cF {\varphi_0}_n |^2 \, d \xi \, dt.
$$
Thus $\{\varphi_n\}$ is a Cauchy sequence in $\cP_0^{(d)}$, so there
is a $\varphi \in \cP_0^{(d)}$ such that $\varphi_n \to \varphi$ in
$\cP_0^{(d)}$. To prove $J(\varphi) = \varphi_0$, we only need to
prove that
$$
I^{2k}(\varphi_n) = {\varphi_0}_n,
$$
which follows from
$$
\cF \left( I^{2k} (\varphi_n) \right) = |\xi|^{-2k} \cF \varphi_n.
$$
\end{proof}

From the above result it follows that:

\begin{corollary}\label{corollary20070927_01}
For $\varphi$, $\eta \in \cP_0^{(d)}$, let $\varphi_0 := J(\varphi)$
and $\eta_0 : = J(\eta)$. Then
$$
\langle \varphi, \eta \rangle_0 = \langle \varphi_0, \eta_0
\rangle_{L_2((0,T) \times \bR^d)}.
$$
\end{corollary}

Now we set
$$
\varphi_1(t,x) = I^{2k} \left( \varphi_0(t,\cdot) \right) (x) =
c_{d,k} \int_{\bR^d} \frac{\varphi_0(t,y)}{|x-y|^{d-2k}} \, dy,
$$
where $\varphi_0 \in L_2((0,T)\times\bR^d)$. Notice that $\varphi_1$
is a measurable function of $(t,x) \in (0,T) \times \bR^d$. Since
$\varphi_0(t,x) \in L_2(\bR^d)$ for a.e. $t \in (0,T)$,
$$
\varphi_1(t,\cdot) \in L_q(\bR^d), \quad \quad \| \varphi_1(t,\cdot)
\|_{L_q(\bR^d)} \le N \| \varphi_0(t,\cdot) \|_{L_2(\bR^d)},
$$
for a.e. $t \in (0,T)$, where $1/q = 1/2 - 2k/d$. This implies that
$$
\varphi_1 \in L_{2,q}((0,T) \times \bR^d),
$$
$$
\| \varphi_1 \|_{L_{2,q}((0,T) \times \bR^d)} = \left(\int_0^T \|
\varphi_1(t,\cdot) \|_{L_q(\bR^d)}^2 \, dt \right)^{1/2} \le N \|
\varphi_0 \|_{L_2((0,T) \times \bR^d)}.
$$

From the $L_p$-theory of parabolic equations with mixed norms, there
exists a unique function $w \in W_{2,q}^{1,2}((0,T) \times \bR^d)$
satisfying
\begin{equation}                            \label{eq01}
w_t - \Delta w = \varphi_1, \qquad w(0,\cdot) = 0.
\end{equation}

\begin{lemma}
\label{lemma-w-equal-h} Let $h(t,x) = \bE M(\varphi) u(t,x)$, where
$u$ is defined in \eqref{solution-process}. Let $\varphi_0 =
J(\varphi)$, $\varphi_1$ be a function defined as above, and $w$ be
the solution to \eqref{eq01}. Then
$$
w = h.
$$
\end{lemma}

\begin{proof}
For $\eta \in \cD((0,T) \times \bR^d)$, find a function $\phi$, a
unique solution to
$$
- \phi_t - \Delta \phi = \eta, \qquad \phi(T,\cdot) = 0.
$$
Then by Lemma \ref{key-calculation}
$$
\langle h, \eta \rangle_{L_2((0,T)\times\bR^d)} = \langle \varphi,
\phi \rangle_0
$$
and
$$
\langle w, \eta \rangle_{L_2((0,T)\times\bR^d)} = \langle w_t -
\Delta w, \phi \rangle_{L_2((0,T)\times\bR^d)} = \langle \varphi_1,
\phi \rangle_{L_2((0,T)\times\bR^d)}.
$$
Find a sequence $\{ {\varphi_0}_n \} \subset \cD((0,T) \times
\bR^d)$ such that ${\varphi_0}_n \to \varphi_0$ in $L_2((0,T) \times
\bR^d)$. Let $\psi_n = I^{2k} ({\varphi_0}_n)$ and $w_n$ be the
solution to $(\partial/\partial t - \Delta) w_n = \psi_n$ and
$w_n(0,\cdot) = 0$. Then
$$
\langle w, \eta \rangle_{L_2((0,T)\times\bR^d)} = \lim_{n \to
\infty} \langle w_n, \eta \rangle_{L_2((0,T)\times\bR^d)} = \lim_{n
\to \infty} \langle \psi_n, \phi \rangle_{L_2((0,T)\times\bR^d)}
$$
$$
= \lim_{n \to \infty} \langle I^{2k}({\varphi_0}_n), \phi
\rangle_{L_2((0,T)\times\bR^d)} = \lim_{n \to \infty} \int_0^T
\int_{\bR^d} \cF {\varphi_0}_n(t,\xi) \overline{\cF \phi(t, \xi)}
|\xi|^{-2k} \, d \xi \, dt.
$$
On the other hand,
$$
\langle h, \eta \rangle_{L_2((0,T)\times\bR^d)} = \langle \varphi,
\phi \rangle_0 = \lim_{n \to \infty} \int_0^T \int_{\bR^d} \cF
\varphi_n(t,\xi) \overline{\cF \phi(t,\xi)} |\xi|^{-4k} \, d \xi \,
dt
$$
$$
= \lim_{n \to \infty} \int_0^T \int_{\bR^d} \cF {\varphi_0}_n(t,\xi)
\overline{\cF \phi(t,\xi)} |\xi|^{-2k} \, d \xi \, dt,
$$
where $\varphi_n = (-\Delta)^{2k} {\varphi_0}_n$. This finishes the
proof.
\end{proof}

Now we prove the main result of this subsection.

\begin{theorem}\label{main-theorem-Riesz}
Suppose that $f$ is the Riesz kernel of order $\alpha=4k,k \in
\bN_{+}$ and $d=4k+1$. Then the process solution $u$ of the
stochastic heat equation (\ref{heat-eq}) with vanishing initial
conditions is locally germ Markov.
\end{theorem}

\begin{proof}
Again we need to verify conditions (i) and (ii) in Theorem
\ref{kunsch-theorem}.

We first verify condition (i).

Let $h(t,x) = \bE M(\varphi) u(t,x)$ and $g(t,x) = \bE M(\eta)
u(t,x)$, where $\varphi, \eta \in \cP_0^{(d)}$, such that
$\text{supp} \, h \cap \text{supp} \, g = \emptyset$. We need to
show that $\langle \varphi, \eta \rangle_0 = 0$. Let $\varphi_0 =
J(\varphi)$ and $\eta_0 = J(\eta)$, where $J$ is the operator
defined in Lemma \ref{lemma20070927_01}. Also let $\varphi_1 =
I^{2k}\left(\varphi_0(t,\cdot)\right)$ and $\eta_1 =
I^{2k}\left(\eta_0(t,\cdot)\right)$. For a.e. $t \in (0, T)$, we
have
$$\varphi_0(t,x)=\bD^{2k} \varphi_1(t,x) =(-\Delta)^k \varphi_1,
$$
where the second equality follows from Corollary
\ref{claim_2006_12_19}. Let $\phi(t,x) \in \cD((0,T) \times \bR^d)$
such that $\phi(t,x) = 0$ on $\text{supp} \, \varphi_1$. Then for
a.e. $t \in (0,T)$,
$$
\int_{\bR^d} \varphi_0(t,x) \phi(t,x) \, dx = \int_{\bR^d}
\varphi_1(t,x) (-\Delta)^k \phi(t,x) \, dx.
$$
Thus
$$
\int_0^T \int_{\bR^d} \varphi_0(t,x) \phi(t,x) \, dx \, dt =
\int_0^T \int_{\bR^d} \varphi_1(t,x) (-\Delta)^k \phi(t,x) \, dx \,
dt = 0.
$$
This shows that $\text{supp} \, \varphi_0 \subset \text{supp} \,
\varphi_1$. Similarly, we have $\text{supp} \, \eta_0 \subset
\text{supp} \, \eta_1$. Hence we have $\langle \varphi_0, \eta_0
\rangle_{L_2((0,T)\times\bR^d)} = 0$. This and Corollary
\ref{corollary20070927_01} prove that $\langle \varphi, \eta
\rangle_0 = 0$.

We now verify condition (ii).

Assume that $\zeta = h + g$, where $\zeta(t,x) = \bE M(\nu) u(t,x)$,
$\text{supp} \, h \cap \text{supp} \, g = \emptyset$, and
$\text{supp} \, h$ is compact. We have to prove that $h \in \cH^u$.

Let $\chi$ be an infinitely differentiable function defined on
$[0,T] \times \bR^d$ such that $0 \le \chi \le 1$, $\chi = 1$ on
$\text{supp} \, \mu$, and $\chi = 0$ in an (relative) open set
containing $\text{supp} \, \nu$. Then $\chi \zeta = h$ and $h \in
W_{2,q}^{1,2}((0,T) \times \bR^d)$. Let $(\partial/\partial t -
\Delta) \zeta = \nu_1$, where $\nu_1 = I^{2k}(\nu_0)$ and $\nu_0 =
J(\nu)$. Now we set
$$
\varphi_1 := \chi \nu_1, \quad \eta_1 := (1-\chi) \nu_1, \quad
\varphi_0 := \chi \nu_0.
$$
Since $\varphi_0 \in L_2((0,T) \times \bR^d)$ there exists a
sequence $\{ {\varphi_0}_n \} \subset \cD((0,T) \times \bR^d)$ such
that ${\varphi_0}_n \to \varphi_0$ in $L_2((0,T) \times \bR^d)$. Let
$\varphi_n = (-\Delta)^{k} {\varphi_0}_n$. Then $\varphi_n \in
\cD((0,T) \times \bR^d)$ and $\{\varphi_n\}$ is a Cauchy sequence in
$\cP_0^{(d)}$ due to the fact that
$$
\cF \varphi_n(t,\xi) = |\xi|^{2k} \cF {\varphi_0}_n(t,\xi).
$$
Let $\varphi$ be the limit in $\cP_0^{(d)}$ of $\varphi_n$. Then
$J(\varphi) = \varphi_0$. Now let $\hat{\varphi}_1 =
I^{2k}(\varphi_0)$ and $\hat{h}(t,x) = \bE M(\varphi) u(t,x)$.

We now prove that $h=\hat h$, which will imply that $h \in \cH^u$.
For this, it suffices to prove that $\varphi_1 = \hat{\varphi}_1$.
By Lemma \ref{lemma-w-equal-h}, we have $(\partial/\partial t -
\Delta) \hat{h} = \hat{\varphi}_1$. Notice that $\nu_1 = \varphi_1 +
\eta_1$ and $\text{supp} \, \varphi_1 \cap \text{supp} \, \eta_1 =
\emptyset$. Thus for a.e. $t \in (0,T)$, we have
$$
\text{supp}_x \, \varphi_1(t,\cdot) \cap \text{supp}_x \, \eta_1(t,
\cdot) = \emptyset.
$$
Thus by Lemma \ref{lemma01} it follows that, for a.e. $t \in (0,T)$,
$$
\bD^{2k} \varphi_1(t,\cdot) = \varphi_0(t,\cdot).
$$
We also have
$$
\bD^{2k} \hat{\varphi}_1 (t,\cdot) = \varphi_0(t,\cdot)
$$
for a.e. $t \in (0,T)$. Since $\varphi_1(t,\cdot) -
\hat{\varphi}_1(t,\cdot) \in I^{2k}(L_2(\bR^d))$ for a.e. $t \in
(0,T)$, there exists a function $f_t(x)$ such that
$$
\varphi_1(t,\cdot) - \hat{\varphi}_1(t,\cdot) = I^{2k}(f_t)
$$
for a.e. $t \in (0,T)$. From the fact that $\bD^{2k}$ is the left
inverse of the $I^{2k}$, we know that
$$
0 = \bD^{2k} \left( \varphi_1(t,\cdot) - \hat{\varphi}_1(t,\cdot)
\right) = \bD^{2k} I^{2k}(f_t) = f_t.
$$
We also know that $I^{2k}(0) = 0$. Hence $\varphi_1(t,\cdot) =
\hat{\varphi}_1(t,\cdot)$ for a.e. $t \in (0,T)$. Therefore,
$\varphi_1 = \hat{\varphi}_1$ as elements in $L_{2,q}((0,T) \times
\bR^d)$. The theorem is proved.
\end{proof}

\appendix

\section{Proof of Relation (\ref{conv-norm0})}
\noindent We need to prove that
$$\int_0^T \int_{\bR^d}
|\cF \varphi_n(s,\xi)-\cF \varphi(s,\xi) |^{2}\mu(d\xi)ds \to 0, \ \
\ {\rm as} \ n \to \infty.$$

\noindent Note that $\cF \varphi_n(s,\xi) = \sum_{m \in I_n} c_d
|Q^{(n)}_m| \exp \{- \mathrm{i} \xi \cdot x^{(n)}_m - (t^{(n)}_m-s)
|\xi|^2\} \eta(t^{(n)}_m,x^{(n)}_m)$ converges pointwise to $$\cF
\varphi(s,\xi)= \int_s^T \int_{\bR^d} c_d \exp \{ - \mathrm{i} \xi
\cdot x - (t-s) |\xi|^2 \} \eta(t,x) \, dx \, dt.$$

\noindent In order to apply the dominated convergence theorem, we
need to prove that there exists a function $\Psi$, which is
square-integrable with respect to $ds \times \mu(d\xi)$ such that
$|\cF \varphi_n(s,\xi)| \le \Psi(s,\xi)$. Let $N$ be a constant such
that $|\eta| \le N$. Then
$$|\cF \varphi_n(s,\xi)| \le \sum_{m \in I_n} c_d |Q^n_m| |\eta(t^{(n)}_m,x^{(n)}_m)|
\exp \{- (t^{(n)}_m-s) |\xi|^2\}$$
$$\le \sum_{m \in I_n} c_d N
|Q^{(n)}_m| e^{- (t^{(n)}_m-s) |\xi|^2} \le \int_s^T \int_{K'} c_d N
e^{-(t-s)|\xi|^2} \, dx \, dt$$
$$\le c_d N |K'| (T-s) 1_{|\xi|< 1} + c_d N |K'|
\frac{1-e^{(s-T)|\xi|^2}}{|\xi|^2} 1_{|\xi| \ge 1}: = \Psi(s,\xi),$$
where $K'$ is a compact set containing $K$, and the third inequality
above is due to the fact that $1_{Q^{(n)}_m}(t,x) e^{-(t_m^{(n)}-s)
|\xi|^2 } \le 1_{Q^{(n)}_m}(t,x) e^{-(t-s) |\xi|^2 }$. Clearly
$\Psi$ is square-integrable with respect to $ds \times \mu (d \xi)$.
This concludes the proof of (\ref{conv-norm0}).

\section{An Auxiliary Lemma}
The following technical result was used in the proof of Theorem
\ref{main-theorem-Riesz} for the verification of condition (ii).

\begin{lemma}
\label{lemma01} Let $2 < d/ 2k$ and $1/q = 1/2 - 2k/d$. Assume that
$\kappa \in I^{2k}(L_2(\bR^d))$ and $\kappa = \mu + \nu$, where
$\mu$, $\nu \in L_q(\bR^d)$, $\text{supp} \, \mu \cap \text{supp} \,
\nu = \emptyset$, and $\text{supp} \, \mu$ is compact. Then
$$
\mu \in I^{2k}(L_2(\bR^d)), \qquad \bD^{2k} \mu = \chi
\bD^{2k}\kappa,
$$
where $\chi$ is an infinitely differentiable function such that $0
\le \chi \le 1$, $\chi = 1$ on $\text{supp} \, \mu$, and $\chi = 0$
in an open set containing $\text{supp} \, \nu$.
\end{lemma}

\begin{proof}
Since $\kappa \in I^{2k}(L_2(\bR^d))$, we have
$$
\bD_{\varepsilon}^{2k} \kappa := \frac{1}{c(d,l,2k)} \int_{|z| >
\varepsilon} \frac{\left( \Delta_z^{l} \kappa
\right)(x)}{|z|^{d+2k}} \, dz \to \bD^{2k} \kappa \quad \text{in}
\quad L_2(\bR^d).
$$
Note that $\bD_{\varepsilon}^{2k} \kappa \in L_2(\bR^d)$ (see
Proposition 2.4 in \cite{almeida-samko06}). Let
$$
\Lambda_{\varepsilon}^{2k} \kappa := \frac{1}{c(d,l,2k)}
\int_{\varepsilon < |z| \le \varepsilon_0} \frac{\left( \Delta_z^{l}
\kappa \right)(x)}{|z|^{d+2k}} \, dz, \quad \varepsilon <
\varepsilon_0.
$$
Since $\Lambda_{\varepsilon}^{2k} \kappa = \bD_{\varepsilon}^{2k}
\kappa - \bD_{\varepsilon_0}^{2k} \kappa$, we have
$\Lambda_{\varepsilon}^{2k} \kappa \in L_2(\bR^d)$. Moreover,
$$
\Lambda_{\varepsilon_1}^{2k} \kappa - \Lambda_{\varepsilon_2}^{2k}
\kappa = \int_{\varepsilon_1 < |z| \le \varepsilon_2} \frac{\left(
\Delta_z^{l} \kappa \right)(x)}{|z|^{d+2k}} \, dz =
\bD_{\varepsilon_1}^{2k} \kappa - \bD_{\varepsilon_2}^{2k} \kappa.
$$
Form this and the fact that $\bD_{\varepsilon}^{2k} \kappa$ is
Cauchy in $L_2(\bR^d)$, it follows that
$\{\Lambda_{\varepsilon}^{2k} \kappa\}$ is a Cauchy sequence in
$L_2(\bR^d)$. Consider $\bD^{2k}_{\varepsilon} (\chi \kappa)$, which
is well-defined because $\chi \kappa \in L_q(\bR^d)$. Note that
$$
\bD^{2k}_{\varepsilon} (\chi \kappa) = \Lambda^{2k}_{\varepsilon}
(\chi \kappa) + \int_{|z| \ge \varepsilon_0} \frac{\left(
\Delta_z^{l} \chi \kappa \right)(x)}{|z|^{d+2k}} \, dz
$$
and
\begin{equation}
\label{technical-eq-appendixC} \Lambda^{2k}_{\varepsilon} (\chi
\kappa) = \chi \Lambda^{2k}_{\varepsilon} \kappa
\end{equation}
for a sufficiently small $\varepsilon_0 > 0$. The proof of
\eqref{technical-eq-appendixC} is stated below. Also note that
$$
\Lambda^{2k}_{\varepsilon} (\chi \kappa) \in L_2(\bR^d), \quad
\int_{|z| \ge \varepsilon_0} \frac{\left( \Delta_z^{l} \chi \kappa
\right)(x)}{|z|^{d+2k}} \, dz \in L_2(\bR^d),
$$
where the latter follows from the fact that $\chi \kappa \in
L_2(\bR^d)$ ($\chi$ has a compact support and $2 \le q$). In
addition, $\Lambda^{2k}_{\varepsilon} \kappa$ approaches a function
in $L_2(\bR^d)$ as $\varepsilon \searrow 0$. Thus
$\bD^{2k}_{\varepsilon} (\chi \kappa)$ converges a function in
$L_2(\bR^d)$. This shows that $\mu = \chi \kappa \in
I^{2k}(L_2(\bR^d))$. To prove $\bD^{2k} \mu = \chi \bD^{2k}\kappa$,
we can just use Corollary \ref{claim_2006_12_19}.

Let us prove \eqref{technical-eq-appendixC} as follows. For each $ x
\in \text{supp} \, \mu$, find $\delta_x > 0$ such that $B(x,
\delta_x) \cap \text{supp} \, \nu = \emptyset$. Consider
$$
\{ B(x, \delta_x/5) : x \in \text{supp} \, \mu \}.
$$
Since $\text{supp} \, \mu$ is compact, we have a finite subset $\{
x_i : i \in I \} \subset \text{supp} \, \mu$ such that
$$
\bigcup_{i \in I} B(x_i, \delta_{x_i}/5) \supset \text{supp} \, \mu.
$$
Let
$$
\cO_1 := \bigcup_{i \in I} B(x_i, \delta_{x_i}/4), \quad \cO_2 :=
\bigcup_{i \in I} B(x_i, \delta_{x_i}/3), \quad \cO_3 := \bigcup_{i
\in I} B(x_i, \delta_{x_i}/2).
$$
Then
$$
\cO_i \supset \text{supp} \, \mu, \quad i = 1, 2, 3
$$
and
$$
\cO_i \cap \text{supp} \, \nu = \emptyset, \quad i = 1, 2, 3.
$$
We may assume that $\chi = 1$ on $\cO_2$ and $\chi = 0$ on
$\cO_3^c$. We also assume that $\varepsilon_0 < \delta/(20 l)$,
where $\delta = \min_{i \in I} \delta_{x_i}$. By definition we have
$$
\Lambda_{\varepsilon}^{2k} \left( \chi \kappa  \right) (x)=
\frac{1}{c} \int_{\varepsilon < |z| \le \varepsilon_0} \frac{\left(
\Delta_z^{l} (\chi \kappa) \right)(x)}{|z|^{d+2k}} \, dz
$$
$$
= \frac{1}{c} \int_{\varepsilon < |z| \le \varepsilon_0} \frac{
\sum_{k=0}^l (-1)^k c_l^k \chi(x-kz) \kappa(x-kz)}{|z|^{d+2k}} \, dz
$$
Let $x \in \cO_1$, i.e. $x \in B(x_i, \delta_{x_i}/4)$. Then
$$
|x_i - (x - k z)| \le |x_i - x| + |k z| \le \delta_{x_i}/4 + l
\delta/(20 l)
$$
$$
\le \delta_{x_i}/4 + \delta_{x_i}/12 = \delta_{x_i}/3.
$$
Thus $x-kz \in \cO_2$. This shows that $\chi(x-kz) = 1$ and
$$
\Lambda_{\varepsilon}^{2k} \left( \chi \kappa  \right) (x)= \chi(x)
\Lambda_{\varepsilon}^{2k} \kappa (x), \quad x \in \cO_1.
$$
Now let $x \in \cO_1^c$. Then for every $x_i$, $i \in I$,
$$
|x_i - (x - k z)| \ge |x_i - x| - l|z|
$$
$$
\ge \delta_{x_i}/4 - l \delta/(20 l) \ge \delta_{x_i}/4 -
\delta_{x_i}/20 = \delta_{x_i}/5.
$$
This implies that $x-kz \notin \text{supp} \, \mu$. In that case we
have
$$
\chi(x-kz) \kappa(x-kz) = 0
$$
because $\kappa(x-kz) = 0$ in case $\chi(x-kz) > 0$. Thus
$\Lambda_{\varepsilon}^{2k} \left( \chi \kappa  \right) (x) = 0$. In
addition, $\chi(x) \Lambda_{\varepsilon}^{2k} \kappa (x) = 0$
(recall that $x \in \cO_1^c$ and $x-kz \notin \text{supp} \, \mu$)
because $\kappa(x-kz) = 0$ in case $\chi(x) > 0$. Indeed, if $x \in
\cO_3 \setminus \cO_1$, then there exists $x_j$, $j \in I$, such
that $x \in B(x_j, \delta_{x_j}/2)$. Then
$$
|x_j - (x-kz)| \le |x_j - x| + l |z| < \delta_{x_j}.
$$
Thus by the choice of $\delta_{x_j}$ it follows that $x-kz \notin
\text{supp} \, \nu$, that is, $\kappa(x-kz) = 0$.
\end{proof}

\par\bigskip\noindent
{\bf Acknowledgment.} The authors are grateful to an anonymous
referee who carefully read the article and made some suggestions for
future research.

\bibliographystyle{amsplain}

\begin{thebibliography}{99}

\bibitem{almeida-samko06} Almeida, A. and Samko, S. G.:
Characterization of Riesz and Bessel potentials on variable Lebesgue
spaces. {\em J. Funct. Spaces Appl.} {\bf 4} (2006) 113--144.

\bibitem{balan-ivanoff02} Balan, R. M. and Ivanoff, B. G.:
A Markov property for set-indexed processes, {\em J. Theoret. Prob.}
{\bf 15} (2002) 553--588.

\bibitem{dalang99} Dalang, R. C.: Extending martingale
measure stochastic integral with applications to spatially
homogenous s.p.d.e.'s. {\em Electr. J. Probab.} {\bf 4}  (1999) no.
6, 1--29. Erratum in {\em Electr. J. Probab.} {\bf 4} (1999) no. 6,
1--5.

\bibitem{dalang-frangos98} Dalang, R. C. and Frangos, N. E.:
The stochastic wave equation in two spatial dimensions. {\em Ann.
Probab.} {\bf 26} (1998) 187--212.

\bibitem{dalang-hou97} Dalang, R. C. and Hou, Q.: On Markov
properties of L\'{e}vy waves in two dimensions. {\em Stoch. Proc.
Appl.} {\bf 72}  (1997) 265--287.

\bibitem{dalang-mueller03} Dalang, R. C. and Mueller, C.:
Some non-linear s.p.d.e.'s that are second order in time. {\em
Electr. J. Probab.} {\bf 8}  (2003) no. 1, 1--21.


\bibitem{donatimartin-nualart94} Donati-Martin, C. and Nualart, D.:
Markov property for elliptic stochastic partial differential
equations. {\em Stoch. Stoch. Rep.} {\bf 46} (1994) 107--115.

\bibitem{evans98} Evans, L. C. (1998) {\em Partial Differential
Equations}. Graduate Studis in Mathematics. {\bf 19}, Amer. Math.
Soc. Providence, RI.

\bibitem{ferrante-sanzsole06} Ferrante, M. and Sanz-Sole, M.:
SPDE's with colored noise: analytic and stochastic approaches. {\em
ESAIM: Prob. Stat.} {\bf 10} (2006) 380--405.

\bibitem{folland95} Folland, G. B.: {\em Introduction to Partial Differential
Equations}, Second Edition, Princeton University Press, Princeton,
1995.

\bibitem{gelfand-vilenkin64} Gel'fand, I. M. and Vilenkin, N. Y.:
{\em Generalized Functions}, Vol. 4, Academic Press, New York, 1964.

\bibitem{krylov99} Krylov, N. V.: An analytic approach to
SPDEs, in: {\em Stochastic partial differential equations: six
perspectives}, {\bf 64} (1999) {\em Math. Surveys Monogr.},
185--242, Amer. Math. Soc., Providence, RI.

\bibitem{krylov01} Krylov, N. V.:
The heat equation in {$L\sb q((0,T),L\sb p)$}-spaces with weights.
{\em SIAM J. Math. Anal.}, {\bf 32} (2001) 1117--1141 (electronic).


\bibitem{kunsch79} Kunsch, H.: Gaussian Markov random
fields. {\em J. Fac. Sci. Univ. Tokyo. Sec. IA}. {\bf 26} (1979)
53--73.


\bibitem{mueller97} Mueller, C.: Long time existence for the
wave equation with a noise term. {\em Ann. Probab.} {\bf 25} (1997)
133--151.

\bibitem{nualart-pardoux94} Nualart, D. and Pardoux, E.:
Markov field properties of solutions of white noise driven
quasi-linear parabolic p.d.e.'s. {\em Stoch. Stoch, Rep.} {\bf 48}
(1994) 17--44.

\bibitem{nualart-sanz79} Nualart, D., and Sanz, M.: A
Markov property for two-parameter Gaussian processes. {\em
Stochastica} {\bf 3} (1979) 1--16.

\bibitem{pitt71} Pitt, L. D.: A Markov property for Gaussian
processes with multidimansional parameters. {\em Arch. Rational
Mech. Anal.} {\bf 43} (1971) 367--391.

\bibitem{pitt-robeva03} Pitt, L. D. and Robeva, R. S.: On the
sharp Markov property for Gaussian random fields and spectral
synthesis in spaces of Bessel potentials. {\em Ann. Probab.} {\bf
31} (2003) 1338--1376.


\bibitem{rozanov82} Rozanov, Yu. A.: {\em Markov Random
Fields}, Springer, Berlin, 1982.

\bibitem{samko76} Samko, S. G.: On spaces of Riesz
potentials. {\em Izv. Acad. Nauk SSSR}, {\bf 40} (1976) 1089--1117.

\bibitem{schwartz66} Schwartz, L.: {\em Th\'{e}orie des
distributions}, Hermann, Paris, 1966.

\bibitem{stein70} Stein, E.M.: {\em Singular Integrals and
Differentiability Properties of Functions}, Princeton University
Press, Princeton, 1970.

\bibitem{walsh86} Walsh, J. B.: An introduction to stochastic
partial differential equations. {\em Ecole d'Et\'{e} de
Probabilit\'{e}s de Saint-Flour XIV. Lecture Notes in Math.} {\bf
1180} (1986) 265--439. Springer-Verlag, Berlin.




\end{thebibliography}

\end{document}